\documentclass[12pt]{amsart}

\usepackage[utf8]{inputenc}
\usepackage[margin=1.25in]{geometry}    % page size and margins
\usepackage{amsfonts,amsthm,amsmath,amssymb,mathrsfs}
\usepackage{pgf}                      % for colors
\usepackage[colorlinks=true,citecolor=magenta,linkcolor=blue,urlcolor=blue]{hyperref}
\usepackage{tikz}
\usepackage[font=scriptsize]{caption}
\usepackage[shortlabels]{enumitem}
\usepackage[capitalize,noabbrev,nameinlink]{cleveref}

\newcommand{\ZZ}{\mathbb{Z}}

\newcommand{\standout}[1]{\textbf{#1}}    % style for terms defined in text

\tikzset{blueedge/.style={blue,ultra thick},greenedge/.style={green!65!black,ultra thick,densely dotted},pinkedge/.style={magenta,ultra thick,dashed}}

\theoremstyle{plain}
\newtheorem{theorem}{Theorem}[section]
\newtheorem{proposition}[theorem]{Proposition}
\newtheorem{lemma}[theorem]{Lemma} 
\newtheorem{corollary}[theorem]{Corollary}

\theoremstyle{definition}
\newtheorem{example}[theorem]{Example}
\newtheorem{remark}[theorem]{Remark}
\newtheorem{defn}[theorem]{Definition}

\newtheorem{construction}[theorem]{Construction}  

\begin{document}

%%%%%%%%%%%%%%%%%%%%%%%%%%%%%%%%%%%%%%%%%%%%%%%%%%%%%%%%%%%%%%%%%%
% METADATA
\title[The Middle Cone]{Metric Dimension of a Direct Product of Three Complete Graphs: The Middle Cone Family}
\author{Briana Foster-Greenwood}
\address{Department of Mathematics and Statistics, California State Polytechnic University, Pomona, CA 91768}
\email{brianaf@cpp.edu}
\author{Christine Uhl}
\address{Department of Mathematics, St.\ Bonaventure University,
St\ Bonaventure, NY 14778}
\email{cuhl@sbu.edu}
\date{\today}
% 05C12 Distance in graphs
% 05C15 Coloring of graphs and hypergraphs
% 05C25 Graphs and abstract algebra (groups, rings, fields, etc.)
% 05B30 Other designs, configurations
% 05B05 Combinatorial aspects of block designs
% 05B15 Orthogonal arrays, Latin squares, Room squares
% 05C65 Hypergraphs
% 05C69 Vertex subsets with special properties (dominating sets, independent sets, cliques, etc.)
\begin{abstract}
In previous work, we determined the metric dimension for a direct product of three isomorphic complete graphs. Turning to the case where the complete graphs may have different orders, there are three families we refer to as the upper, lower, and middle cones.
We determine the metric dimension and location-total-domination number for a family of direct products of three complete graphs 
stemming from the middle cone.
We explicitly describe minimum resolving sets. 
To verify the sets are resolving, we define a basic landmark system and show it will be a resolving set if and only if its associated 3-edge-colored hypergraph avoids three types of forbidden subgraphs. This generalizes the technique used for three isomorphic factors.
\end{abstract}

\maketitle
%%%%%%%%%%%%%%%%%%%%%%%%%%%%%%%%%%%%%%%%%%%%%%%%%%%%%%%%%%%%%%

\section{Introduction}
Let $G$ be a finite connected graph with vertex set $V$. For vertices $x,y\in V$, define the distance $d(x,y)$ to be the length of a shortest path between $x$ and $y$ in $G$. Given a subset of vertices $W\subseteq V$, whose elements are referred to as \standout{landmarks}, we say $W$ is a \standout{resolving set} (or $W$ \standout{resolves} $G$) provided that for every pair of distinct vertices $x,y\in V-W$, there exists a landmark $w\in W$ such that $d(x,w)\neq d(y,w)$. A \standout{minimum resolving set} for $G$ is a resolving set of minimum size. The \standout{metric dimension} of $G$, denoted $\dim(G)$, is the size of a minimum resolving set.

Metric dimension has applications to GPS, chemical structures, gene sequencing, pattern recognition, and robot navigation (see the surveys \cite{Tillquist2021} and  \cite{Kuziak2021}).
Alhough the decision variant of the problem of finding metric dimension is NP-complete, advances have been made for families of graphs, including various types of product graphs. For example, the metric dimension is known for Cartesian products of two complete graphs \cite{Caceres2007}, three isomorphic complete graphs \cite{DrewesJager2020}, and three arbitrary complete graphs \cite{GledelJager2024}. Bounds and asymptotic results are known for Hamming graphs $K_a^{\square k}$ (a Cartesian product of $k$ isomorphic complete graphs), which have connections to coin-weighing and the game Mastermind (see the survey \cite{Tillquist2021} for discussion and references). Metric dimension of lexicographic product graphs is studied in \cite{Jannesari2012}.

Due to their connection to unitary Cayley graphs \cite[Theorem 3.2]{Sander2010}, we focus on direct products of complete graphs.
Given two graphs $G$ and $H$, the \standout{direct product graph} $G\times H$ has vertex set $V(G\times H)=V(G)\times V(H)$ and component-wise adjacency. That is, $(g_1,h_1)$ is adjacent to $(g_2,h_2)$ in $G\times H$ if and only if $g_1$ is adjacent to $g_2$ in $G$ and $h_1$ is adjacent to $h_2$ in $H$. 

Results on metric dimension of direct products with two factors include when the factors are cycles or paths \cite{Vetrik2017} or complete graphs \cite{Kuziak2017}. Considering three factors, in \cite[Theorem 17]{FGUhl}, we found for $n\geq 5$, the metric dimension of a direct product of three isomorphic complete graphs is
\[
\dim(K_n \times K_n \times K_n) = 2n-1. 
\]
In this paper, we continue this investigation but allow for the factors to be nonisomorphic.

For a direct product of three complete graphs $K_{n_1}\times K_{n_2}\times K_{n_3}$, we identify the vertex set of $K_{n_i}$ with the set $\{1,2,\ldots,n_i\}$ so that the vertex set of  $K_{n_1}\times K_{n_2}\times K_{n_3}$ is $\{(x_1,x_2,x_3):1\leq x_i\leq n_i\text{ for $i=1,2,3$}\}$. Vertices $(x_1,x_2,x_3)$ and  $(y_1,y_2,y_3)$ are adjacent if and only if $x_i\neq y_i$ for all $i\in\{1,2,3\}$. For notational convenience, we write $\mathbf{n}=(n_1,n_2,n_3)$ and let $K(\mathbf{n})=K_{n_1}\times K_{n_2}\times K_{n_3}$. At times we will increment or decrement the parameter by the all ones vector $\mathbf{1}=(1,1,1)$.

% introduce cones
Let $\mathcal{N}=\{(n_1,n_2,n_3)\in\mathbb{N}^3:3\leq n_1,n_2\leq n_3\}$.
To organize our work, we partition the parameter space $\mathcal{N}$ into ``cones''.
 Define the \standout{lower cone} to be the set of all $\mathbf{n}\in \mathcal{N}$ such that $2n_3<3\max(n_1,n_2)$; the \standout{middle cone} to be the set of all $\mathbf{n}\in \mathcal{N}$ such that $3\max(n_1,n_2)\leq 2n_3\leq n_1n_2$; and the \standout{upper cone} to be the set of all $\mathbf{n}\in \mathcal{N}$ such that $2n_3>n_1n_2$.

For $\mathbf{n}\in\mathcal{N}$, the restriction $n_i\geq 3$ guarantees $K(\mathbf{n})$ is a connected graph of diameter two.
Indeed, for any pair of vertices $\alpha=(a_1,a_2,a_3)$ and $\beta=(b_1,b_2,b_3)$, there is a common neighbor $\gamma=(c_1,c_2,c_3)$ with $c_i\notin\{a_i,b_i\}$ for each $i$. Thus,
\[
d(\alpha,\beta)=\begin{cases}
    0 & \text{if $a_i=b_i$ for all $i$} \\
    1 & \text{if $a_i\neq b_i$ for all $i$} \\
    2 & \text{if $a_i=b_i$ for some (but not all) $i$.} 
\end{cases}
\]

By \cite[Theorem 2]{FGUhl}, we know that the metric dimension of $K(\mathbf{n})$ is at least $2n_3-1$. Our main conclusion of the current paper is that we get equality whenever $\mathbf{n-1}$ is in the middle cone.
This also tells us the adjacency dimension since metric dimension equals the adjacency dimension for graphs of diameter two \cite{Jannesari2012}.

The remainder of the paper is organized as follows.
In \cref{sec:hypandforbidden}, we recall from \cite{FGUhl} how to represent a set of landmarks by an edge-colored hypergraph and prove an equivalent characterization of resolving sets in terms of landmark graphs. This extends \cite[Theorem 9]{FGUhl} to the case where the landmark graph has hyperedges containing more than two vertices. In \cref{sec:constructions}, we describe a concrete construction of resolving sets.
In \cref{sec:proofs}, we prove our main result \cref{mdimthm} by showing the sets from our construction extend to minimum resolving sets whose landmark graphs avoid the forbidden configurations identified in \cref{thm:poofyloopyverboten}. 
Finally, in \cref{sec:conclusion}, we comment on the upper and lower cones and propose some open problems. 

%%%%%%%%%%%%%%%%%%%%%%%%%%%%%%%%%%%%%%%%%%%%%%%%%%%%%%%
%%%%%%%%%%%%%%%%%%%%%%%%%%%%%%%%%%%%%%%%%%%%%%%%%%%%%%%
\section{Landmark Graphs and Forbidden Configurations}\label{sec:hypandforbidden}

As in \cite{FGUhl}, we convert the problem of determining whether a set of landmarks resolves $K(\mathbf{n})$ into a problem of determining whether a related edge-colored hypergraph avoids forbidden subgraphs.
This section parallels \cite[Section 3]{FGUhl}, generalizing to allow for landmark graphs with hyperedges that have more than two vertices and identifying a new set of forbidden subgraphs accordingly. First, we establish some hypergraph terminology.

\subsection{Hypergraphs}

A \standout{hypergraph} consists of a nonempty set of vertices, say $W$, and a multiset of \standout{hyperedges} (\standout{edges} for brevity), each of which is a nonempty subset of $W$. We classify hyperedges by how many vertices they contain. In particular, a \standout{loop} is a hyperedge with exactly one vertex; a \standout{stick} is a hyperedge with exactly two vertices; and a \standout{poofy edge} is a hyperedge with at least three vertices. Our definition of hypergraph allows for multiple hyperedges with the same set of vertices. For instance, we may have a \standout{triple loop}, which is three loops with the same vertex.

\subsection{Landmark Graphs}
In order to determine if $W$, a subset of vertices of the graph $K(\mathbf{n})$, resolves the graph, we use the \textbf{landmark graph} $\mathcal{G}(W)$ as in \cite{FGUhl}.  Given $W$, let $W_{i,a}$ be the set of elements in $W$ whose $i$-th coordinate is $a$.  Then $\mathcal{G}(W)$ is the hypergraph with vertex set $W$ and hyperedges those $W_{i,a}$ that are nonempty.  
The landmark graph $\mathcal{G}(W)$ has a proper $3$-edge coloring by assigning color $i$ to the sets $W_{i,a}$.  
We refer to the first color as blue, second as green, and third as pink.  See \cref{fig:landmarkgraph346} for an example of a landmark graph (which represents a minimum resolving set of $K(4,5,7)$ by \cref{thm:poofyloopyverboten}).

% example of a landmark graph
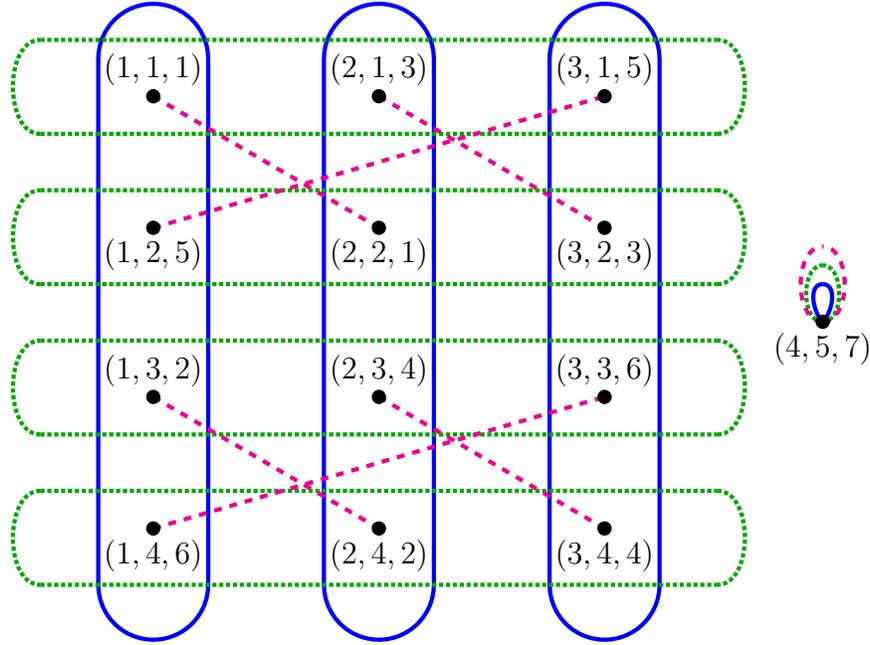
\begin{figure}
    \centering
\begin{tikzpicture}[scale=1]
    \draw[blueedge,double distance=40pt, line cap=round] (0,-2.5)--(0,4.5);
    \draw[blueedge,double distance=40pt, line cap=round] (6,-2.5)--(6,4.5);
    \draw[blueedge,double distance=40pt, line cap=round] (3.0,-2.5)--(3.0,4.5);
   
    \draw[greenedge] (-1.5,4.75)--(7.5,4.75);    \draw[greenedge] (-1.5,3.5)--(7.5,3.5);
    \draw[greenedge]   (-1.5,3.5) to[out=180,in=180] (-1.5,4.75);
    \draw[greenedge]   (7.5,3.5) to[out=0,in=0] (7.5,4.75);
    
    \draw[greenedge] (-1.5,2.75)--(7.5,2.75);
    \draw[greenedge] (-1.5,1.5)--(7.5,1.5);
    \draw[greenedge]   (-1.5,1.5) to[out=180,in=180] (-1.5,2.75);
    \draw[greenedge]   (7.5,1.5) to[out=0,in=0] (7.5,2.75);

     \draw[greenedge] (-1.5,.75)--(7.5,.75);
    \draw[greenedge] (-1.5,-.5)--(7.5,-.5);
    \draw[greenedge]   (-1.5,-.5) to[out=180,in=180] (-1.5,.75);
    \draw[greenedge]   (7.5,-.5) to[out=0,in=0] (7.5,.75);

    \draw[greenedge] (-1.5,-2.5)--(7.5,-2.5);
    \draw[greenedge] (-1.5,-1.25)--(7.5,-1.25);
    \draw[greenedge]   (-1.5,-1.25) to[out=180,in=180] (-1.5,-2.5);
    \draw[greenedge]   (7.5,-1.25) to[out=0,in=0] (7.5,-2.5);

    \draw[pinkedge] (0,4)--(3,2.25);
    \draw[pinkedge] (0,2.25)--(6,4);
    \draw[pinkedge] (0,0)--(3,-1.75);
    \draw[pinkedge] (0,-1.75)--(6,0);
    \draw[pinkedge] (3,4)--(6,2.25);
    \draw[pinkedge] (3,0)--(6,-1.75);
    
    \draw[fill=black] (0,4) circle (2.5pt) node[above] {$(1,1,1)$};
    \draw[fill=black] (0,2.25) circle (2.5pt) node[below] {$(1,2,5)$};
    \draw[fill=black] (0,0) circle (2.5pt) node[above] {$(1,3,2)$};
    \draw[fill=black] (0,-1.75) circle (2.5pt) node[below] {$(1,4,6)$};
      
    \draw[fill=black] (3.,4) circle (2.5pt) node[above] {$(2,1,3)$};
    \draw[fill=black] (3,2.25) circle (2.5pt) node[below] {$(2,2,1)$};
    \draw[fill=black] (3,0) circle (2.5pt) node[above] {$(2,3,4)$};
    \draw[fill=black] (3,-1.75) circle (2.5pt) node[below] {$(2,4,2)$};

    \draw[fill=black] (6,4) circle (2.5pt) node[above] {$(3,1,5)$};
    \draw[fill=black] (6,2.25) circle (2.5pt) node[below] {$(3,2,3)$};
    \draw[fill=black] (6,0) circle (2.5pt) node[above] {$(3,3,6)$};
    \draw[fill=black] (6,-1.75) circle (2.5pt) node[below] {$(3,4,4)$};

     \draw[blueedge] (8.9,1) to[out=135,in=180] (8.9,1.5);
    \draw[blueedge] (8.9,1.0) to[out=45,in=0] (8.9,1.5);
    \draw[greenedge] (8.9,1.0) to[out=180,in=180] (8.9,1.75);
    \draw[greenedge] (8.9,1.0) to[out=0,in=0] (8.9,1.75);
    \draw[pinkedge] (8.9,1.0) to[out=180,in=180] (8.9,2);
    \draw[pinkedge] (8.9,1) to[out=0,in=0] (8.9,2);
    \draw[fill=black] (8.9,1) circle (2.5pt) node[below] {$(4,5,7)$};
\end{tikzpicture}
    \caption{Landmark graph $\mathcal{G}(W)$ of a minimum resolving set for $K(4,5,7)$. Landmarks in the same blue (solid) hyperedge have the same first coordinate; landmarks connected by a green (dotted) edge have the same second coordinate; and landmarks connected by a pink (dashed) edge have the same third coordinate.}
    \label{fig:landmarkgraph346}
\end{figure}

In \cite[Definition 4]{FGUhl}, we defined a $2$-basic landmark system. Here, we generalize by allowing for hyperedges with more than two vertices.
\begin{defn}\label{def:basic}
    Let $\mathbf{n}=(n_1,n_2,n_3)\in\mathcal{N}$. For $W$ a subset of vertices of the graph $K(\mathbf{n})$, we say $W$ is a \standout{basic landmark system} provided that in the landmark graph $\mathcal{G}(W)$:
\begin{enumerate}
    \item there is a a full set of hyperedges in each color, i.e., for $1\leq i\leq 3$, there are $n_i$ hyperedges of color $i$; %(to ensure $W_{i,a}$ nonempty)
    \item  each hyperedge contains at least two vertices; and
    \item hyperedges of different colors intersect in at most one vertex.
\end{enumerate}
Letting $u=(n_1+1,n_2+1,n_3+1)$, we extend a basic landmark system for $K(\mathbf{n})$ to a \standout{triple-looped landmark system} $W\cup\{u\}$ for $K(\mathbf{n}+\mathbf{1})$. 
\end{defn}

\begin{remark}\label{rem:triple-looped-properties}
   Observe that the landmark graph $\mathcal{G}(W\cup\{u\})$ has the hyperedges from $\mathcal{G}(W)$ as well as a new triple-loop, so
   we continue to use the notation $W_{i,a}$,
   extending to include the loops $W_{1,n_1+1}=W_{2,n_2+1}=W_{3,n_3+1}=\{u\}$. 
   The landmark graph of $W\cup\{u\}$ still satisfies properties (1) and (3) of the definition of a basic landmark system but does not satisfy property (2) due to the presence of loops.
\end{remark}

\subsection{Footprints}We now recall from \cite{FGUhl} the definition of footprints and how to use them to determine if a basic landmark system is a resolving set.
Given a basic landmark system $W$ and vertex $\alpha = (a_1,a_2,a_3)$ of $K(\mathbf{n})$, the \textbf{footprint of $\alpha$} (relative to $W$) is the subgraph of $\mathcal{G}(W)$ induced by the set of edges $\{W_{1,a_1},W_{2,a_2},W_{3,a_3}\}$. The elements of $W_{1,a_1}\cup W_{2,a_2}\cup W_{a_3}$ are the vertices \standout{covered} by the footprint of $\alpha$. By the first condition of a basic landmark system, the footprint of $\alpha$ consists of exactly three hyperedges (one of each color).  If $\alpha \in W$, the edges in the footprint of $\alpha$ intersect in a common vertex.  (If $W$ is not a basic landmark system, we can still define footprints, but a footprint may have fewer than three edges, or possibly be empty.)

By \cite[Remark 6]{FGUhl}, distinct vertices $\alpha$ and $\beta$ not in $W$ are resolved if and only if their footprints cover different sets of vertices. Indeed, for $w\in W$,  if $w$ is in the footprint of $\alpha$ but not in the footprint of $\beta$, then $d(\alpha,w)=2$ and $d(\beta,w)=1$; and if $w$ is in both footprints or neither, then $d(\alpha,w)=d(\beta,w)$.

\subsection{Forbidden configurations for basic landmark systems}
There are three forbidden subgraphs such that if a landmark graph contains them, we immediately know that the basic landmark system is not resolving.  They are the bad $4$-cycle, plain hex, and shark teeth as defined below.  See \cref{tab:forbiddenfootprints} for a visual representation.

\begin{defn}
    A \standout{bad $4$-cycle} is a tri-colored $4$-cycle in which the opposite edges of the same color are sticks. (The remaining edges may be sticks or poofy.) A \standout{plain hex} is a $6$-cycle in which every edge is a stick and the edge colors follow a pattern blue-green-pink-blue-green-pink. A \standout{rainbow 2-2-triangle} is a path of two sticks of different colors where the initial and final vertices (``termini'') of the path belong to the same (possibly poofy) edge of a third color. \standout{Shark teeth} are two vertex disjoint rainbow $2$-$2$-triangles whose termini belong to the same poofy edge.
\end{defn}

We begin by showing that if a landmark graph contains one of the three forbidden subgraphs, then the corresponding set of landmarks is not resolving.

\begin{lemma}\label{lem:BADstuff}
     Let $\mathbf{n}\in\mathcal{N}$ and let $W$ be a subset of vertices of the graph $K(\mathbf{n})$. If the landmark graph $\mathcal{G}(W)$ contains a bad $4$-cycle, plain hex, or shark teeth, then $W$ does not resolve the graph $K(\mathbf{n})$. 
\end{lemma}

\begin{proof} We consider each forbidden configuration. For concreteness, we specify the colors of the edges, but these may be permuted.

    \textit{Bad $4$-cycle.} The proof of \cite[Lemma 7]{FGUhl} still holds with allowing the opposite edges of two different colors to be poofy.  See also \cref{tab:forbiddenfootprints} for a proof by picture.
    
     \textit{Plain hex.} See the proof of \cite[Lemma 8]{FGUhl}. 

    \textit{Shark teeth.}  Suppose the graph of $W$ has shark teeth (illustrated in \cref{tab:forbiddenfootprints}) with green sticks $W_{2,a_2}=\{u_1,u_2\}$ and $W_{2,b_2}=\{v_1,v_2\}$; pink sticks $W_{3,b_3}=\{u_2,u_3\}$ and $W_{3,a_3}=\{v_2,v_3\}$; and blue edge $W_{1,a_1}$ containing $\{u_1,u_3,v_1,v_3\}$. Then
    \[W_{1,a_1}\cup W_{2,a_2}\cup W_{3,a_3}=W_{1,a_1}\cup W_{2,b_2}\cup W_{3,b_3},\]
    so the footprints of $\alpha=(a_1,a_2,a_3)$ and $\beta=(a_1,b_2,b_3)$ cover the same set of vertices. Moreover, $\alpha$ and $\beta$ are not landmarks. Hence $\alpha$ and $\beta$ are not resolved. 
\end{proof}

\begin{table}
       \centering
        {\renewcommand{\arraystretch}{1.4}\begin{tabular}{|c|c|c|}
        \hline
        \textbf{Forbidden Subgraph }& \textbf{Footprint 1} & \textbf{Footprint 2}\\
         \hline
    \hline
    bad $4$-cycle  &&\\ %(poofy)
       \begin{tikzpicture}[scale=1.75]
     \draw[blueedge,double distance=15pt,line cap=round] (-.25,1)--(1.25,1);
    \draw[greenedge] (-.35,.15)--(1.35,0.15);
    \draw[greenedge] (-.35,-.15)--(1.35,-0.15);
    \draw[greenedge] (-.35,.15) to[out=180,in=180] (-.35,-.15);
    \draw[greenedge] (1.35,0.15)  to[out=0,in=0] (1.35,-0.15);
    \draw[pinkedge] (0,1)--(0,0);
    \draw[pinkedge] (1,0)--(1,1);
     \draw[fill=white] (0,0) circle (2pt);
     \draw[fill=white] (0,1) circle (2pt);
     \draw[fill=white] (1,1) circle (2pt);
     \draw[fill=white] (1,0) circle (2pt);
\end{tikzpicture}& \begin{tikzpicture}[scale=1.75]
     \draw[blueedge,double distance=15pt,line cap=round] (-.25,1)--(1.25,1);
       \draw[greenedge] (-.35,.15)--(1.35,0.15);
    \draw[greenedge] (-.35,-.15)--(1.35,-0.15);
    \draw[greenedge] (-.35,.15) to[out=180,in=180] (-.35,-.15);
    \draw[greenedge] (1.35,0.15)  to[out=0,in=0] (1.35,-0.15);
    \draw[pinkedge] (0,1)--(0,0);
     \draw[fill=white] (0,0) circle (2pt);
     \draw[fill=white] (0,1) circle (2pt);
     \draw[fill=white] (1,1) circle (2pt);
     \draw[fill=white] (1,0) circle (2pt);
\end{tikzpicture} &\begin{tikzpicture}[scale=1.75]
     \draw[blueedge,double distance=15pt,line cap=round] (-.25,1)--(1.25,1);
       \draw[greenedge] (-.35,.15)--(1.35,0.15);
    \draw[greenedge] (-.35,-.15)--(1.35,-0.15);
    \draw[greenedge] (-.35,.15) to[out=180,in=180] (-.35,-.15);
    \draw[greenedge] (1.35,0.15)  to[out=0,in=0] (1.35,-0.15);
    \draw[pinkedge] (1,0)--(1,1);
     \draw[fill=white] (0,0) circle (2pt);
     \draw[fill=white] (0,1) circle (2pt);
     \draw[fill=white] (1,1) circle (2pt);
     \draw[fill=white] (1,0) circle (2pt);
\end{tikzpicture}\\
    \hline
    plain hex &&\\
    \begin{tikzpicture}[scale=1.75] % forbidden 6-cycle
        \draw[blueedge] (60:.75)--(120:.75);
        \draw[blueedge] (240:.75)--(300:.75);
        \draw[greenedge] (0:.75)--(60:.75);
        \draw[greenedge] (180:.75)--(240:.75);
        \draw[pinkedge] (300:.75)--(360:.75);
        \draw[pinkedge] (120:.75)--(180:.75);
        \foreach \x in {1,...,6}
            \draw[fill=white] ({60*\x}:.75) circle (2pt);
    \end{tikzpicture}\hfill
    &\begin{tikzpicture}[scale=1.75] % forbidden 6-cycle
        \draw[blueedge] (240:.75)--(300:.75);
        \draw[greenedge] (0:.75)--(60:.75);
        \draw[pinkedge] (120:.75)--(180:.75);
        \foreach \x in {1,...,6}
            \draw[fill=white] ({60*\x}:.75) circle (2pt);
    \end{tikzpicture}\hfill&\begin{tikzpicture}[scale=1.75] % forbidden 6-cycle
        \draw[blueedge] (60:.75)--(120:.75);
        \draw[greenedge] (180:.75)--(240:.75);
        \draw[pinkedge] (300:.75)--(360:.75);
        \foreach \x in {1,...,6}
            \draw[fill=white] ({60*\x}:.75) circle (2pt);
    \end{tikzpicture}\hfill\\
    \hline
     shark teeth &&\\
    \begin{tikzpicture}[scale=1.1]
    \draw[blueedge,double distance=15pt,line cap=round] (-.25,1)--(3.25,1);
    \draw[pinkedge] (1,1)--(.5,0);
    \draw[pinkedge] (2.5,0)--(3,1);
    \draw[greenedge] (2.5,0)--(2,1);
    \draw[greenedge] (.5,0)--(0,1);
    \draw[fill=white] (0,1) circle (2pt);
    \draw[fill=white] (1,1) circle (2pt) ;
    \draw[fill=white] (2,1) circle (2pt);
    \draw[fill=white] (3,1) circle (2pt);
    \draw[fill=white] (.5,0) circle (2pt) ;
    \draw[fill=white] (2.5,0) circle (2pt);
\end{tikzpicture}&
       \begin{tikzpicture}[scale=1.1]
    \draw[blueedge,double distance=15pt,line cap=round] (-.25,1)--(3.25,1);
    \draw[pinkedge] (2.5,0)--(3,1);
    \draw[greenedge] (.5,0)--(0,1);
    \draw[fill=white] (0,1) circle (2pt);
    \draw[fill=white] (1,1) circle (2pt) ;
    \draw[fill=white] (2,1) circle (2pt);
    \draw[fill=white] (3,1) circle (2pt);
    \draw[fill=white] (.5,0) circle (2pt) ;
    \draw[fill=white] (2.5,0) circle (2pt);
\end{tikzpicture}&
      \begin{tikzpicture}[scale=1.1]
    \draw[blueedge,double distance=15pt,line cap=round] (-.25,1)--(3.25,1);
    \draw[pinkedge] (1,1)--(.5,0);
    \draw[greenedge] (2.5,0)--(2,1);
    \draw[fill=white] (0,1) circle (2pt);
    \draw[fill=white] (1,1) circle (2pt) ;
    \draw[fill=white] (2,1) circle (2pt);
    \draw[fill=white] (3,1) circle (2pt);
    \draw[fill=white] (.5,0) circle (2pt) ;
    \draw[fill=white] (2.5,0) circle (2pt);
\end{tikzpicture}\\
\hline
 triangle union triple loop &&\\
 \begin{tikzpicture}[scale=1.75] % rainbow triangle
        \draw[blueedge,double distance=15pt,line cap=round] (-1,-.35)--(1,-.35);
        \draw[greenedge] ({330+120}:.75)--({330+240}:.75);
        \draw[pinkedge] ({330+120}:.75)--({330}:.75);
        \foreach \x in {1,2,3}
            \draw[fill=white] ({330+120*\x}:.75) circle (2pt);
              \draw[blueedge] (1,0) to[out=135,in=180] (1,0.5);
    \draw[blueedge] (1,0) to[out=45,in=0] (1,0.5);
    \draw[greenedge] (1,0) to[out=180,in=180] (1,0.75);
    \draw[greenedge] (1,0) to[out=0,in=0] (1,0.75);
    \draw[pinkedge] (1,0) to[out=180,in=180] (1,1);
    \draw[pinkedge] (1,0) to[out=0,in=0] (1,1);
    \draw[fill=white] (1,0) circle (2pt);
    \end{tikzpicture}&
     \begin{tikzpicture}[scale=1.75] % rainbow triangle
        \draw[blueedge,double distance=15pt,line cap=round] (-1,-.35)--(1,-.35);
        \draw[greenedge] ({330+120}:.75)--({330+240}:.75);
        \foreach \x in {1,2,3}
            \draw[fill=white] ({330+120*\x}:.75) circle (2pt);
    \draw[pinkedge] (1,0) to[out=180,in=180] (1,1);
    \draw[pinkedge] (1,0) to[out=0,in=0] (1,1);
    \draw[fill=white] (1,0) circle (2pt);
    \end{tikzpicture} & \begin{tikzpicture}[scale=1.75] % rainbow triangle
        \draw[blueedge,double distance=15pt,line cap=round] (-1,-.35)--(1,-.35);
        \draw[pinkedge] ({330+120}:.75)--({330}:.75);
        \foreach \x in {1,2,3}
            \draw[fill=white] ({330+120*\x}:.75) circle (2pt);
    \draw[greenedge] (1,0) to[out=180,in=180] (1,0.75);
    \draw[greenedge] (1,0) to[out=0,in=0] (1,0.75);
    \draw[fill=white] (1,0) circle (2pt);
    \end{tikzpicture} \\
    \hline
        \end{tabular}}
        \caption{Forbidden subgraphs and two different footprints covering the same set of vertices. Different line styles correspond to different edge colors. The edges drawn as ovals could have more vertices than shown.}
        \label{tab:forbiddenfootprints}
    \end{table}
In fact, these are the only obstructions to resolving.
\begin{theorem}\label{lem:poofyverboten}
    Let $\mathbf{n}\in\mathcal{N}$ and let $W$ be a basic landmark system for $K(\mathbf{n})$. Then $W$ is a resolving set if and only if the landmark graph $\mathcal{G}(W)$ avoids bad $4$-cycles, plain hex, and shark teeth.
\end{theorem}
\begin{proof}
  As seen in \cref{lem:BADstuff}, if any of the three forbidden subgraphs occur in $\mathcal{G}(W)$, then $W$ is not resolving.

  Conversely, suppose that $W$ is a basic landmark system whose graph avoids bad $4$-cycles, plain hex, and shark teeth. We will show that $W$ is a resolving set. Let $\alpha=(a_1,a_2,a_3)$ and $\beta=(b_1,b_2,b_3)$ be distinct vertices of $K(\mathbf{n})$ that are not in $W$. We will show their footprints cannot cover the same set of vertices. We consider cases based on the sizes of the hyperedges in the footprints.

  First, note that if the footprints of $\alpha$ and $\beta$ do not have any poofy edges, then by \cite[proof of Theorem 9]{FGUhl}, they do not cover the same set of vertices. For the remainder of the proof, we assume that the footprint of $\alpha$ or $\beta$ has at least one poofy edge.

    There are two main cases.  The first is when the footprint of $\alpha$ contains a poofy edge $W_{i,a_i}$ such that $W_{i,a_i}\neq W_{i,b_i}$.    
    The second is when $W_{i,a_i}=W_{i,b_i}$ for each poofy edge $W_{i,a_i}$ in the footprint of $\alpha$ (and respectively $W_{i,b_i}$ in footprint of $\beta$).  

  We first consider the case where there exists a poofy edge $W_{i,a_i}$ in the footprint of $\alpha$ such that $W_{i,a_i}\neq W_{i,b_i}$. Without loss of generality, say $i=1$.  
  If the footprints of $\alpha$ and $\beta$ cover the same set of vertices, then we must have $W_{1,a_1}\subseteq W_{2,b_2}\cup W_{3,b_3}$ since the blue edges $W_{1,a_1}$ and $W_{1,b_1}$ are disjoint. But by definition of a basic landmark system, the blue edge $W_{1,a_1}$ can intersect the green edge $W_{2,b_2}$ in at most one vertex and likewise for the pink edge $W_{3,b_3}$. This is a contradiction, since it means $|W_{1,a_1}|\leq 2$.

  Next we consider the case when for each $1\leq i\leq 3$, we have $W_{i,a_i}=W_{i,b_i}$ or $|W_{i,a_i}|=|W_{i,b_i}|=2$. We will further break into cases depending on whether the footprints of $\alpha$ and $\beta$ have one or two poofy edges in common. (Three poofy edges in common would make $\alpha=\beta$, a contradiction. No poofy edges in common would mean the edges of the footprints are all sticks, a case that was previously handled.)
  
  We encourage the reader to actively draw pictures and refer to \cref{tab:proofpics} as they are reading the following cases.

  Case 1 (footprints have exactly one poofy edge in common). By symmetry of permuting indices, it suffices to consider the case where $W_{1,a_1}=W_{1,b_1}$ and $|W_{i,a_i}|=|W_{i,b_i}|=2$ for $i=2,3$. By definition of a basic landmark system, the green and pink edges can have at most one endpoint in the blue hyperedge $W_{1,a_1}=W_{1,b_1}$ so must each have an endpoint outside of the blue hyperedge. We will consider subcases depending on how many vertices are outside the blue hyperedge.

  Subcase 1.1 (only one vertex outside the blue hyperedge).
  The footprint of $\alpha$ looks like a rainbow $2$-$2$-triangle, as in \cref{tab:proofpics} [1.1], and the landmark outside the blue hyperedge can only belong to one pink stick and one green stick. So, if the footprint of $\beta$ also covers the outside vertex then $\beta=\alpha$.

  Subcase 1.2 (exactly two vertices outside the blue hyperedge).  In this case, we show that the constraints of a basic landmark system force shark teeth, which is a contradiction.    
  Say $u$ and $v$ are the vertices outside the blue hyperedge and the footprint of $\alpha$ has pink stick $uw_1$ and green stick $vw_2$ for some $w_1,w_2\in W$  with $w_1\neq w_2$ since $\alpha$ is not a landmark (see \cref{tab:proofpics}~[1.2]). Since $\alpha\neq\beta$, the footprint of $\beta$ must have pink stick $vw_3$ and green stick $uw_4$ for some landmarks $w_3,w_4\in W$ with $w_3\neq w_4$. We claim that $\{w_1,w_2\}$ and $\{w_3,w_4\}$ are disjoint: If $w_1=w_3$ or $w_2=w_4$, then we get a contradiction to the proper $3$-edge-coloring of $\mathcal{G}(W)$. If $w_1=w_4$ or $w_2=w_3$, then we have pink and green hyperedges intersecting in more than one vertex, which contradicts $W$ a basic landmark system. Hence $w_1,w_2,w_3,w_4$ are distinct, and we have shark teeth (see \cref{tab:forbiddenfootprints}), a contradiction. 
  
  Subcase 1.3 (exactly three vertices outside the blue hyperedge). For the footprint of $\alpha$ to cover all vertices outside the blue hyperedge, there must be one stick outside (covering two outside vertices) and one stick with a vertex in and vertex out (covering the third outside vertex), as in \cref{tab:proofpics}~[1.3] (possibly with colors permuted). The same must be true for $\beta$. If we start with a drawing of $\alpha$'s footprint and then try to add in $\beta$'s footprint while maintaining a proper coloring, we are forced to have two sticks of different colors with the same pair of endpoints, which violates property (3) of a basic landmark system (\cref{def:basic}), a contradiction. 
  
  Subcase 1.4 (four vertices outside the blue hyperedge). If the footprints of $\alpha$ and $\beta$ cover the same set of vertices and there are four vertices outside the blue hyperedge, then the green and pink edges in the footprint of $\alpha$ must form a perfect matching outside the blue hyperedge (see \cref{tab:proofpics}~[1.4]), and likewise, the green and pink edges for $\beta$ form a perfect matching outside the blue hyperedge. But this is impossible unless $\alpha=\beta$ because the landmark graph is properly colored and hyperedges of different colors intersect in at most one vertex.

  Case 2 (footprints have exactly two poofy edges in common). By symmetry of permuting indices, it suffices to consider the case where $W_{1,a_1}=W_{1,b_1}$, $W_{2,a_2}=W_{2,b_2}$, and $|W_{3,a_3}|=|W_{3,b_3}|=2$. We consider subcases depending how many vertices are outside the union of the blue and green edges.

  Subcase 2.1 (no vertices outside). 
  In this case, the pink sticks $W_{3,a_3}$ and $W_{3,b_3}$ would each have one endpoint in $W_{1,a_1}=W_{1,b_1}$ and the other endpoint in $W_{2,a_2}=W_{2,b_2}$. However, this creates a bad $4$-cycle, which we are assuming we avoided.
  
  Subcase 2.2 (one or two vertices outside). The only way for the footprints of $\alpha$ and $\beta$ to both cover the vertex/vertices outside the blue and green hyperedges is if the pink sticks $W_{3,a_3}$ and $W_{3,b_3}$ coincide. But then $\alpha=\beta$, a contradiction.

  We have shown that in all cases the footprints of $\alpha$ and $\beta$ do not cover the same set of vertices. Thus $W$ is a resolving set.
\end{proof}

\begin{table}
       \centering
        {\renewcommand{\arraystretch}{1.4}\begin{tabular}{|c|c|c|}
        \hline
        %\textbf{First Main Case }
        \textbf{Case:} &\textbf{1.1} & \textbf{1.2}\\
        \hline
        &&\\
&
    \begin{tikzpicture}[scale=1.2]
    \draw[blueedge,double distance=15pt,line cap=round] (-.25,1)--(3.25,1);
        \draw[pinkedge] (1.5,0)--(3,1);
    \draw[greenedge] (1.5,0)--(2,1);
    \draw[fill=white] (0,1) circle (2pt);
    \draw[fill=white] (2,1) circle (2pt);
    \draw[fill=white] (3,1) circle (2pt);
    \draw[fill=white] (1.5,0) circle (2pt) ;
\end{tikzpicture}&
   \begin{tikzpicture}[scale=1.2]
    \draw[blueedge,double distance=15pt,line cap=round] (-.25,1)--(3.25,1);
         \draw[pinkedge] (2.5,0)--(3,1);
     \draw[greenedge] (.5,0)--(1,1);
    \draw[fill=white] (0,1) circle (2pt);
    \draw[fill=white] (1,1) circle (2pt) ;
    \draw[fill=white] (2,1) circle (2pt);
    \draw[fill=white] (3,1) circle (2pt);
    \draw[fill=white] (.5,0) circle (2pt) ;
    \draw[fill=white] (2.5,0) circle (2pt);
\end{tikzpicture}
\\
&&\\
\hline
\textbf{Case:} & \textbf{1.3} &\textbf{1.4} \\
\hline
&&\\
&     \begin{tikzpicture}[scale=1.2]
    \draw[blueedge,double distance=15pt,line cap=round] (-.25,1)--(3.25,1);
    \draw[pinkedge] (2.5,0)--(1.5,0);
    \draw[greenedge] (.5,0)--(1,1);
    \draw[fill=white] (0,1) circle (2pt);
    \draw[fill=white] (1,1) circle (2pt) ;
    \draw[fill=white] (3,1) circle (2pt);
    \draw[fill=white] (.5,0) circle (2pt) ;
    \draw[fill=white] (2.5,0) circle (2pt);
    \draw[fill=white] (1.5,0) circle (2pt);
\end{tikzpicture}& \begin{tikzpicture}[scale=1.2]
    \draw[blueedge,double distance=15pt,line cap=round] (-.25,1)--(3.25,1);
    \draw[pinkedge] (0,0)--(1,0);
    \draw[greenedge] (2,0)--(3,0);
    \draw[fill=white] (0,1) circle (2pt);
    \draw[fill=white] (1,1) circle (2pt) ;
    \draw[fill=white] (3,1) circle (2pt);
    \draw[fill=white] (0,0) circle (2pt) ;
    \draw[fill=white] (1,0) circle (2pt);
    \draw[fill=white] (2,0) circle (2pt);
        \draw[fill=white] (3,0) circle (2pt);
\end{tikzpicture}\\ 
& & \\
\hline
\end{tabular}}
 \caption{Illustrations for possible footprints of $\alpha$ (up to permuting colors) in Cases 1.1-1.4 of the proof of \cref{lem:poofyverboten}. The edges drawn as ovals could have more vertices than shown. Different line styles correspond to different edge colors.}
        \label{tab:proofpics}
\end{table}

\subsection{Forbidden configurations for triple-looped landmark systems}
The following lemma generalizes \cite[Lemma 10]{FGUhl} from triangles to $2$-$2$-triangles.

\begin{lemma}[$2$-$2$-triangle union triple loop]\label{ex:tripleloop}
Let $\mathbf{n}\in\mathcal{N}$ and let $W$ be a subset of vertices of the graph $K(\mathbf{n})$.  If the hypergraph $\mathcal{G}(W)$ contains a triple-loop and a rainbow $2$-$2$-triangle, then $W$ is not a resolving set.
\end{lemma}
 \begin{proof} 
 The proof of \cite[Lemma 10]{FGUhl} still holds with $W_{1,a_1}=\{w_1,w_2\}$ replaced by $W_{1,a_1}\supseteq\{w_1,w_2\}$ to allow for rainbow $2$-$2$-triangles. See also \cref{tab:forbiddenfootprints} for a proof by picture.
 \end{proof}

The next theorem provides conditions for when a resolving set for $K(\mathbf{n})$ can be extended to a resolving set for $K(\mathbf{n}+\mathbf{1})$ by adding a triple loop.

\begin{theorem}\label{thm:poofyloopyverboten}
Let $\mathbf{n}\in\mathcal{N}$ and suppose $W$ is a basic landmark system for $K(\mathbf{n})$. Let $u=(n_1+1,n_2+1,n_3+1)$. Then the triple-looped landmark system $W\cup\{u\}$ resolves $K(\mathbf{n}+\mathbf{1})$ if and only if the landmark graph $\mathcal{G}(W)$ avoids the following forbidden subgraphs:
\begin{enumerate}
    \item bad $4$-cycles,
    \item plain hexes, and
    \item rainbow $2$-$2$-triangles\footnote{Note that avoidance of rainbow $2$-$2$-triangles also implies avoidance of shark teeth.}.
\end{enumerate}
\end{theorem}

\begin{proof}
    Let $V(\mathbf{n})$ and $V(\mathbf{n}+\mathbf{1})$ be the vertex sets of the graphs $K(\mathbf{n})$ and $K(\mathbf{n}+\mathbf{1})$, respectively.
    If a bad $4$-cycle or plain hex exists in the graph of $W$, then $W$ does not resolve $K(\mathbf{n})$, and so $W\cup\{u\}$ does not resolve $K(\mathbf{n}+\mathbf{1})$. This is because the new landmark $u$ is adjacent (distance $1$) to all of the vertices in $V(\mathbf{n})$, so vertices in $V(\mathbf{n})$ can only be resolved using landmarks from $W$.
    If a rainbow $2$-$2$-triangle exists in the graph of $W$, then a $2$-$2$-triangle and triple loop exist in the graph of $W\cup\{u\}$, so by \cref{ex:tripleloop}, $W\cup\{u\}$ does not resolve $K(\mathbf{n}+\mathbf{1})$.
    
    Conversely, suppose $W$ is a basic landmark system such that $\mathcal{G}(W)$ avoids bad $4$-cycles, plain hex, and rainbow $2$-$2$-triangles (and hence also shark teeth). We will show that $W \cup \{u\}$ is a resolving set for $K(\mathbf{n}+\mathbf{1})$.  
    
    Suppose $\alpha=(a_1, a_2, a_3)$ and $\beta=(b_1, b_2, b_3)$ are distinct vertices in $K(\mathbf{n}+\mathbf{1})$ and that neither is a landmark. 
    Note that, by \cref{lem:poofyverboten}, $W$ is a resolving set for $K(\mathbf{n})$, so if $\alpha$ and $\beta$ are both in $V(\mathbf{n})$, then they are resolved by some landmark in $W$.  If one vertex, say $\alpha$, is in $V(\mathbf{n})$ and the other vertex, $\beta$, is in $V(\mathbf{n}+\mathbf{1})-V(\mathbf{n})$, then $u$ resolves $\alpha$ and $\beta$, as $u$ is in the footprint of $\beta$, but not in the footprint of $\alpha$. 
    The remaining case, where $\alpha$ and $\beta$ are both in $V(\mathbf{n}+\mathbf{1})-V(\mathbf{n})$, requires the most work. We will show the footprints of $\alpha$ and $\beta$ cannot cover the same set of vertices. 
    
    We will consider cases based on the sizes of the hyperedges in the footprints.
    First note that since $\alpha, \beta \in V(\mathbf{n}+\mathbf{1})-V(\mathbf{n})$, their footprints must each include at least one loop, and since $\alpha$ and $\beta$ are not landmarks, the footprints cannot contain more than two loops.  
    Furthermore, for cases when the footprints consist of only loops and sticks, see \cite[proof of Theorem 11]{FGUhl} for why the footprints of $\alpha$ and $\beta$ cannot cover the same set of vertices.  

    We are going to use the same main cases as in the proof of \cref{lem:poofyverboten}.  The first main case, when the footprint of $\alpha$ contains a poofy edge $W_{i,a_i}$ such that $W_{i,a_i} \neq W_{i,b_i}$ leads to the same contradiction.  The second main case has some different details.
    
  Suppose for each $1\leq i\leq 3$, we have $W_{i,a_i}=W_{i,b_i}$ or $|W_{i,a_i}|=|W_{i,b_i}| \leq 2$. As in the  proof of \cref{lem:poofyverboten}, we will further break into cases depending on whether the footprints of $\alpha$ and $\beta$ have one or two poofy edges in common.  Again, more than two poofy edges would make $\alpha = \beta$, and no poofy edges in common would mean the edges in the footprints are all sticks or loops, a case that was previously handled.

  Case 1 (footprints have exactly one poofy edge in common). By symmetry of permuting indices, it suffices to consider the case where $W_{1,a_1}=W_{1,b_1}$ and $|W_{i,a_i}|=|W_{i,b_i}|\leq 2$ for $i=2,3$. By \cref{rem:triple-looped-properties}, the green and pink edges can have at most one endpoint in the blue hyperedge $W_{1,a_1}=W_{1,b_1}$ so must each have an endpoint outside of the blue hyperedge. We will consider subcases depending on how many vertices are outside the blue hyperedge.
  Note that the possible subcases to consider are one, two, or three vertices outside the blue hyperedge, as both footprints must include a loop.
  
  Subcase 1.1 (only one vertex outside the blue hyperedge). 
  Since the footprints of $\alpha$ and $\beta$ each include a loop, the vertex outside the blue hyperedge must be $u$, the vertex that has a triple loop. It follows that $W_{2,a_2}=W_{2,b_2}= W_{3,a_3}=W_{3,b_3} =\{u\}$, so $\alpha=\beta$, which is a contradiction. 
  
  Subcase 1.2 (exactly two vertices outside the blue hyperedge). In this case, the footprints of $\alpha$ and $\beta$ each consist of a loop, a stick, and the blue hyperedge. We assumed $\alpha\neq\beta$, so up to permuting the green and pink colors, we can assume the loops are $W_{2,b_2}=W_{3,a_3}=\{u\}$ and the sticks are $W_{3,b_3}=\{v,w\}$ and $W_{2,a_2}=\{v,w'\}$, where $v$ is outside the blue hyperedge and $w$ and $w'$ are in the blue hyperedge.
  Since pink and green sticks cannot intersect in more than one vertex, we have $w\neq w'$.
  But this is a contradiction because we now have a rainbow $2$-$2$-triangle with
  green stick $W_{2,a_2}$ and pink stick $W_{3,b_3}$ whose endpoints $w$ and $w'$ are contained in the blue poofy hyperedge $W_{1,a_1}=W_{1,b_1}$.

  Subcase 1.3 (exactly three vertices outside the blue hyperedge). Since the footprints of both $\alpha$ and $\beta$ include at least one loop and need to cover all three vertices outside of blue hyperedge, we have that both footprints have a loop and a stick % a $L_1 \cup P_2$ 
  outside of the blue hyperedge. The loops must be at the same vertex, $u$, which means the sticks must be on the same two vertices. But you can't have different colored sticks between the same two vertices (since different colored hyperedges intersect in at most one vertex), so the footprints of $\alpha$ and $\beta$ must share the same stick, and in turn, the same loop. But then $\alpha=\beta$, a contradiction. 
  
  Case 2 (footprints have exactly two poofy edges in common). By symmetry of permuting indices, it suffices to consider the case where $W_{1,a_1}=W_{1,b_1}$, $W_{2,a_2}=W_{2,b_2}$ and $|W_{3,a_3}|=|W_{3,b_3}|\leq 2$. As the footprints of $\alpha$ and $\beta$ must each contain a loop, 
  we get
  $W_{3,a_3}=W_{3,b_3}=\{u\}$. Thus $\alpha=\beta$, a contradiction.
  
    Thus in all cases the footprints of $\alpha$ and $\beta$ cannot cover the same set of vertices and so $\alpha$ and $\beta$ are resolved.  Hence $W \cup \{u\}$ is a resolving set for the graph $K(\mathbf{n+1})$.
\end{proof}

%%%%%%%%%%%%%%%%%%%%%%%%%%%%%%%%%%%%%%%%%%%%%%%%%%%
%%%%%%%%%%%%%%%%%%%%%%%%%%%%%%%%%%%%%%%%%%%%%%%%%%
\section{Middle Cone Construction}\label{sec:constructions}
We now describe a construction that will allow us to build resolving sets for graphs $K(\mathbf{n})$ for $\mathbf{n}$ in the middle cone and then ``bump up'' to \textit{minimum} resolving sets for $K(\mathbf{n+1})$ by adding a triple loop.

Recall that the middle cone is the set of all $\mathbf{n}\in \mathcal{N}$ such that $3\max(n_1,n_2)\leq 2n_3\leq n_1n_2$.  
At each step in the construction, our choices make it easier to avoid forbidden configurations from \cref{lem:BADstuff}. 
We include examples after the construction to aid the reader in understanding all of the notation required.

\begin{construction}\label{construction}
   Let $\mathbf{n}\in\mathcal{N}$ with $3\max(n_1,n_2)\leq 2n_3\leq n_1n_2$ and let $f=\lfloor n_2/2\rfloor$.
  For $W\subseteq V(K(\mathbf{n}))$, let $\mathcal{M}(W)$ denote the size $3\times |W|$ matrix whose columns give the coordinates of the elements of $W$. For our construction, we will create $W$ as the union of two sets $W^L$ and $W^R$, each with $n_3$ landmarks.
  We say the elements of $W^L$ are \standout{on the left side} and the elements of $W^R$ are \standout{on the right side}.\\

  \noindent\textbf{Even Construction.}
  First we consider the case where $n_2$ is even.

  \textit{Third row.} Use the sequence $1,2,\ldots,n_3$ in the third row of matrix $\mathcal{M}(W^L)$. Do the same for the third row of the matrix $\mathcal{M}(W^R)$.\footnote{This gives us $n_3$ pink sticks in the landmark graph. The endpoints of a pink stick are on opposite sides (left/right).}

  \textit{Second row.} 
  In the middle row of $\mathcal{M}(W^L)$ repeat $1,2,\ldots,n_2/2$ (with the last repetition being only partially complete if $n_3$ is not a multiple of $n_2/2$). For the middle row of $\mathcal{M}(W^R)$, do the same pattern but add $n_2/2$ to each entry.\footnote{We have forced landmarks on opposite sides (left/right) to have different second coordinates, meaning the green hyperedges in the landmark graph only contain landmarks that are on the same side. This restricts the possible paths in the landmark graph and helps avoid forbidden configurations.} 

  \textit{First row.}
  For the first row of $\mathcal{M}(W^L)$, we use a sequence of the form \[1^{(\ell_1)},2^{(\ell_2)},\ldots,n_1^{(\ell_{n_1})},\] 
  where the parenthetical exponents $\ell_i$ denote multiplicities that sum to $n_3$, chosen as described in \cref{rem:multiplicity}.
 For the first row of $\mathcal{M}(W^R)$, we use the same multiplicities but add one (mod $n_1$) to the entries to get
  \[2^{(\ell_1)},3^{(\ell_2)},\ldots,n_1^{(\ell_{n_1-1})},1^{(\ell_{n_1})}.\]
  We refer to the submatrix of $\mathcal{M}(W^L)$ (likewise of $\mathcal{M}(W^R)$) consisting of all columns with the same first (blue) coordinate as a \standout{blue block}.
  The \standout{start of blue block $i$} is the first column number for which $i$ appears in row $1$ of $\mathcal{M}(W^L)$, i.e., column number $1+\sum_{j=1}^{i-1}\ell_j$.\\

 \noindent\textbf{Odd Construction.}
  In the case that $n_2$ is odd, we start with the even construction for $n_2-1$ (without worrying about whether the inequalities hold) and then modify by introducing landmarks that have middle (green) coordinate $n_2$. We refer to these landmarks as \standout{neutral} since they will occur in both $W^L$ and $W^R$. We place the neutral landmarks using the \standout{insert} and \standout{find and replace} operations explained in \cref{rem:moves}. Placement depends on how often the multiplicities $\ell_i$ exceed the floor $f=\frac{n_2-1}{2}$:

Case 1 ($\ell_i>f$ more than once). Insert a pair of neutrals at position $1+\sum_{j=1}^{i-1} \ell_j$ (i.e., the start of blue block $i$) whenever $\ell_i>f$.

Case 2 ($\ell_i> f$ at most once). Insert a pair of neutrals at position $1$ and then find one more landmark in $W^L$ to replace with a neutral.
\end{construction}

In the next two definitions, we explicitly describe the multiplicities and the insert and replace operations needed for our construction. 

\begin{defn}[Multiplicities]\label{rem:multiplicity}
To determine the multiplicities $\ell_1,\ell_2,\ldots,\ell_{n_1}$, we use the division algorithm to write $n_3=qn_1+r$ with $q\in\ZZ$ and $0\leq r\leq n_1-1$. 
  We use multiplicity $q+1$ a total of $r$ times and multiplicity $q$ a total of $n_1-r$ times (so the multiplicities sum to $n_3$) and distribute the multiplicities so that they alternate as much as possible.\footnote{When $q=1$, the alternating ensures poofiness of as many blue hyperedges as possible, helping to prevent creating forbidden configurations that involve sticks of multiple colors.} 
  
  In particular, if $r\leq n_1-r$, we alternate multiplicities $q+1,q$ a total of $r$ times and end with a tail of $q$'s ($n_1-2r$ of them) so the first row of $\mathcal{M}(W^L)$ looks like
  \[1^{(q+1)},2^{(q)},3^{(q+1)},\ldots,(2r-1)^{(q+1)},(2r)^{(q)},(2r+1)^{(q)},\ldots,n_1^{(q)}.\]
  If $r> n_1-r$, we alternate multiplicities $q+1,q$ a total of $n_1-r$ times and end with a tail of $(q+1)$'s ($2r-n_1$ of them) so the first row of $\mathcal{M}(W^L)$ looks like
  \[1^{(q+1)},2^{(q)},3^{(q+1)},\ldots,(2(n_1-r))^{(q)},(2(n_1-r)+1)^{(q+1)},\ldots,n_1^{(q+1)}.\]
  \end{defn}

Next, we elaborate on the ``insert'' and ``find and replace'' operations used in the odd case of \cref{construction}.

\begin{defn}[Insert and Find and Replace]\label{rem:moves}
We say \textbf{insert a pair of neutrals at position $c$} to mean: 
insert $n_2$ in row $2$, column $c$ of $\mathcal{M}(W^L)$, and
insert $n_2$ in row $2$, column $c+1$ of $\mathcal{M}(W^R)$.
Insert does not replace anything; e.g., in $\mathcal{M}(W^L)$, we just shift the row $2$ entries in columns $c$ through $n_3$ to the right one to make room for the neutral coordinate, and the entry in column $n_3$ falls out of matrix.

We say {\bf find a landmark to replace with a neutral} to mean 
choose a landmark (matrix column) of the form $(x,1,z)$ with $x\notin\{1,2,n_1\}$ and change it to $(x,n_2,z)$.\footnote{Our proof of \cref{Lem:onlypinksticks} includes explanation of why such a landmark exists.}
\end{defn}

The next three examples illustrate the cases of the construction.

\begin{example}[Even $n_2$]\label{ex5611}
We demonstrate \cref{construction} for $\mathbf{n}=(5,6,11)$.
Since $n_3=11$, we put the numbers $1$ through $11$ in the bottom row of each matrix. Since $n_2=6$, we split in half and use $1,2,3$ repeating in the middle row of $\mathcal{M}(W^L)$ and $4,5,6$ repeating in the middle row of $\mathcal{M}(W^R)$. Since $n_1=5$, we can write $n_3=qn_1+r$ with $q=2$ and $r=1$. By \cref{rem:multiplicity}, our multiplicities are $3,2,2,2,2$. Thus, in the top row of $\mathcal{M}(W^L)$, we start with $1$'s, repeating $\ell_1=q+1=3$ times; then place $2$'s, repeating $\ell_2=q=2$ times; etc. The top row of $\mathcal{M}(W^R)$ uses the same multiplicities but adds one (mod $n_1=5$) to the entries.
Hence, we have the matrices
   \[\mathcal{M}(W^L)=\left[\begin{array}{ccccccccccc}
      1 & 1 & 1 & 2 & 2 & 3 & 3 & 4 &4 &5 &5 \\
      1 & 2 & 3 &  1 & 2 & 3 & 1 & 2 & 3 &1&2 \\
      1 & 2 & 3 & 4 & 5 & 6 & 7 & 8 & 9 &10  &11
   \end{array}\right]\]
   and 
      \[\mathcal{M}(W^R)=\left[\begin{array}{ccccccccccc}
      2 & 2 & 2 & 3 & 3 & 4 & 4 & 5 & 5 & 1 & 1\\
      4 & 5 & 6 & 4 &5 & 6 & 4&5 & 6 &4&5 \\
      1 & 2 & 3 & 4 & 5 & 6 & 7 & 8 & 9&10&11
   \end{array}\right].\]
\end{example}

%%%%%%%%%%%%%%%%
\begin{example}[Odd $n_2$: Insert and replace]\label{ex5711}
We demonstrate \cref{construction} for $\mathbf{n}=(5,7,11)$. Note that $f=\lfloor 7\rfloor=3$.
We start with the matrices $\mathcal{M}(W^L)$ and $\mathcal{M}(W^R)$ from \cref{ex5611}.  Since $\ell_i \leq f$ always, we insert a pair of neutrals (with middle coordinate $7$) at position $1$ and then choose a landmark of the form $(x,1,z)$ with $x\notin\{1,2,n_1\}$ and change it to $(x,n_2,z)$.  Here we replace $(4,1,8)$ with $(4,7,8)$.  Thus we have the matrices
   \[\mathcal{M}(W^L)=\left[\begin{array}{ccccccccccc}
      1 & 1 & 1 & 2 & 2 & 3 & 3 & 4 &4 &5 &5 \\
     \raisebox{.5pt}{\textcircled{\raisebox{-.9pt} {$7$}}} & 1 & 2 & 3 &  1 & 2 & 3 &   \raisebox{.5pt}{\textcircled{\raisebox{-.9pt} {$7$}}} & 2 & 3 &1 \\
      1 & 2 & 3 & 4 & 5 & 6 & 7 & 8 & 9 &10  &11
   \end{array}\right]\]
   and 
      \[\mathcal{M}(W^R)=\left[\begin{array}{ccccccccccc}
      2 & 2 & 2 & 3 & 3 & 4 & 4 & 5 & 5 & 1 & 1\\
      4 &   \raisebox{.5pt}{\textcircled{\raisebox{-.9pt} {$7$}}} &  5 & 6 & 4 &5 & 6 & 4&5 & 6 &4 \\
      1 & 2 & 3 & 4 & 5 & 6 & 7 & 8 & 9&10&11
   \end{array}\right].\]
   We circle the neutral coordinate $7$ for ease of spotting it. 
\end{example}

\begin{example}[Odd $n_2$: Insert multiple times]\label{5717}
    We demonstrate \cref{construction} for $\mathbf{n}=(5,7,17)$.
Since $n_3=17$, we put the numbers $1$ through $17$ in the bottom row of each matrix. Since $n_2=7$, we split in half and use $1,2,3$ repeating in the middle row of $\mathcal{M}(W^L)$ and $4,5,6$ repeating in the middle row of $\mathcal{M}(W^R)$. Note that $7$ is our neutral coordinate that we use when inserting.  Since $n_1=5$, we can write $n_3=qn_1+r$ with $q=3$ and $r=2$. By \cref{rem:multiplicity}, our multiplicities are $4,3,4,3,3$. Thus, in the top row of $\mathcal{M}(W^L)$, we start with $1$'s, repeating $\ell_1=q+1=4$ times; then place $2$'s, repeating $\ell_2=q=3$ times; then place $3$'s, repeating $\ell_3=q+1=4$ times; etc.  The top row of $\mathcal{M}(W^R)$ uses the same multiplicities, but adds one (mod $n_1=5$) to the entries.
Since $\ell_1, \ell_3 > 3$, pairs of neutrals are inserted at positions $1$ and $8$ (i.e., at the start of the first and third blue blocks).  This prevents the blue edges from intersecting the green edges more than once.
Hence, we have the matrices
   \[\mathcal{M}(W^L)=\left[\begin{array}{ccccccccccccccccc}
      1 & 1 & 1 & 1  & 2 & 2 & 2  & 3 & 3 & 3 & 3 & 4 & 4 & 4  & 5 & 5 & 5  \\
       \raisebox{.5pt}{\textcircled{\raisebox{-.9pt} {$7$}}} & 1 & 2 & 3 & 1 & 2 & 3 &   \raisebox{.5pt}{\textcircled{\raisebox{-.9pt} {$7$}}} & 1 &2 & 3 & 1 & 2 & 3 & 1 & 2 & 3 \\
      1 & 2 & 3 & 4 & 5 & 6 & 7 & 8 & 9 & 10 & 11 & 12 & 13 & 14 & 15 & 16 & 17 
   \end{array}\right]\]
   and 
     \[\mathcal{M}(W^R)=\left[\begin{array}{ccccccccccccccccc}
      2 & 2 & 2 & 2  & 3 & 3 & 3  & 4 & 4 & 4 & 4 & 5 & 5 & 5  & 1 & 1 & 1  \\
      4 &   \raisebox{.5pt}{\textcircled{\raisebox{-.9pt} {$7$}}} & 5 & 6 & 4 & 5 & 6 & 4  &  \raisebox{.5pt}{\textcircled{\raisebox{-.9pt} {$7$}}}&5 & 6 & 4 & 5 & 6 & 4 & 5 & 6 \\
      1 & 2 & 3 & 4 & 5 & 6 & 7 & 8 & 9 & 10 & 11 & 12 & 13 & 14 & 15 & 16 & 17 
   \end{array}\right].\]

\end{example}

\section{Construction Produces Resolving Sets} \label{sec:proofs}
Now that we have introduced our construction, it remains to prove that all steps of our construction are possible and yield a basic landmark system that resolves the graph $K(\mathbf{n})$. We will also prove we can add a triple loop to get a minimum resolving set for $K(\mathbf{n+1})$. We start with a technical lemma about the multiplicities.  

\begin{lemma}\label{lem:f+1}
 Let $\mathbf{n}\in\mathcal{N}$ with $3\max(n_1,n_2)\leq 2n_3\leq n_1n_2$. 
  Let $f=\lfloor n_2/2\rfloor$. Let $(\ell_1,\ell_2,\ldots,\ell_{n_1})$ be the list of multiplicities as defined in \cref{rem:multiplicity}. Then we have the following bounds:
  \begin{enumerate}[label=(\alph*)]
      \item $\ell_1\geq 2$.
      \item If $n_2$ is even, then $\ell_i\leq f$ for all $1\leq i\leq n_1$.
      \item If $n_2$ is odd, then
      $\ell_i\leq f+1$ for all $1\leq i\leq n_1$, and multiplicity $f+1$ does not appear consecutively (i.e., if $\ell_i=f+1$, then $\ell_{i+1}\neq f+1$, with indices interpreted mod $n_1$). 
  \end{enumerate}    
\end{lemma}
\begin{proof}
Part (a): Towards a contradiction assume $\ell_1 = 1$. By \cref{rem:multiplicity}, this gives that $q=1$ and $r=0$, and so $n_1=n_3$.  This is a contradiction to $3\max(n_1,n_2)\leq 2n_3$.

    Part (b): Suppose $n_2$ is even. Writing $n_3=n_1q+r$, as in \cref{rem:multiplicity}, we have 
    \begin{align*}
        2n_3-n_1n_2 & =2(n_1q+r)-n_1(2f) \\
                    & =2n_1(q-f)+2r.
    \end{align*}
    Since $2n_3-n_1n_2\leq 0$, we have $2n_1(q-f)+2r\leq 0$ and $q+r/n_1\leq f$.
    
    Let $1\leq i\leq n_1$. Recall that $0\leq r<n_1$. For the case $r=0$, observe that $q+r/n_1=q$ and $\ell_i=q$, so $\ell_i\leq f$. For the case $0<r<n_1$, observe that $q+r/n_1\leq f$ implies $q+1\leq f$. So, since $\ell_i\in\{q,q+1\}$, we have $\ell_i\leq f$. Hence $\ell_i\leq f$ for all $i$ when $n_2$ is even.
    
     Part (c): Next, suppose $n_2$ is odd. Note that
    \begin{align*}
        2n_3-n_1n_2 & =2(n_1q+r)-n_1(2f+1) \\
                    & =2n_1(q-f)+2r-n_1.
    \end{align*}
    We first show $\ell_i\leq f+1$ for $1\leq i\leq n_1$. By \cref{rem:multiplicity}, $\ell_i\in\{q,q+1\}$ for all $i$, so if we can show $q+1\leq f+1$, then we will know $\ell_i\leq f+1$ for all $i$. Towards a contradiction, suppose $q>f$.
     Then we get $2n_3-n_1n_2\geq 2n_1+2r-n_1= n_1+2r>0$,
     which contradicts our assumption $2n_3\leq n_1n_2$. Hence $q\leq f$ and $q+1\leq f+1$, so $\ell_i\leq f+1$ for all $i$.

     Finally, we show that $f+1$ does not appear consecutively in the list of multiplicities.
     If $q<f$, then $\ell_i\leq q+1\leq f$ for all $i$, so certainly $f+1$ does not appear consecutively.
     If $q=f$, then we get $2n_3-n_1n_2=2r-n_1$, so since $2n_3-n_1n_2\leq 0$, we have $2r-n_1\leq 0$, which means our multiplicities end with a tail of $q=f$'s. It follows that $(f+1)$ does not appear consecutively (including modularly).
\end{proof}

Next we show that the construction results in pink sticks and poofy blue and green edges in the corresponding landmark hypergraph.  This guarantees that the construction has no forbidden plain hex or shark teeth.

\begin{lemma}\label{Lem:onlypinksticks}
 Let $\mathbf{n}\in\mathcal{N}$ with $3\max(n_1,n_2)\leq 2n_3\leq n_1n_2$.
    If $W$ is constructed as in \cref{construction}, then in the landmark graph $\mathcal{G}(W)$ all pink hyperedges are sticks and all blue and green hyperedges are poofy.
\end{lemma}
\begin{proof}
First, note that all pink hyperedges are sticks, as the bottom row of each matrix in \cref{construction} consists of $1,2,\ldots, n_3$.  So, for all $1 \leq i \leq n_3$, the pink edge $W_{3,i}$ consists of two landmarks---one coming from column $i$ of $\mathcal{M}(W^L)$ and the other from column $i$ of $\mathcal{M}(W^R)$.

Next, we consider the blue hyperedges, i.e., the sets of landmarks with the same first coordinate. 
 By \cref{construction}, the sizes of the blue hyperedges are sums of consecutive multiplicities. That is, $|W_{1,i}|=\ell_{i-1}+\ell_i$ for $1\leq i\leq n_1$ (interpret indices mod $n_1$ so $\ell_0=\ell_{n_1}$). 

 Recall that each multiplicity $\ell_i$ is either $q$ or $q+1$.
If $q\geq2$, then $\ell_{i-1}+\ell_{i}\geq 2q\geq 4$ for $1\leq i\leq n_1$, so each blue hyperedge is poofy. 
If $q=1$, then by doubling both sides of $n_3=n_1+r$ and using the hypothesis $3n_1\leq 2n_3$, we get $2r\geq n_1$. Thus, by construction and \cref{rem:multiplicity}, we are in the case where there are no consecutive multiplicities $q$. That is, if $\ell_i=q=1$, then $\ell_{i-1}=\ell_{i+1}=q+1=2$. Hence all blue hyperedges are poofy in the case $q=1$ as well.

Finally, we consider the green hyperedges. If $n_2$ is even, then $3\max{(n_1,n_2)} \leq 2n_3$ implies that $n_3\geq 3(\frac{n_2}{2})$, which implies that the repetition of $1,2, \ldots, n_2/2$ in the second row of $\mathcal{M}(W^L)$ occurs at least three times. Thus, the  green edges $W_{2,i}$ are poofy for $1\leq i\leq f$.  Similarly, the green hyperedges $W_{2,i}$ are poofy for $f+1\leq i\leq n_2$.

Now consider $n_2$ is odd. 
Let $k$ be the number of times that $\ell_i >f$, i.e., the number of elements in the set 
$\{i:1\leq i\leq n_1\text{ and }\ell_i>f\}$.

Case 1 ($k \geq 3$). By construction, we have $k$ neutral landmarks on each side, so there are at least six neutral landmarks, which means the neutral green hyperedge $W_{2,n_2}$ is poofy.
Now we show that the other green hyperedges are also poofy.
Since $k \geq 3$, it means that $\ell_i >f$ (i.e., $\ell_i=f+1$ by \cref{lem:f+1}) three or more times.
So $n_3=\sum_{i=1}^{n_1}\ell_i\geq k(f+1)$ and $n_3-k\geq kf\geq 3f$.
Thus, after ignoring the $k$ neutral landmarks on the left side, we have green coordinates $1,2,\ldots,f$ repeated at least three times, which means the green hyperedges $W_{2,i}$ are poofy for $1\leq i\leq f$. Similarly, the green hyperedges $W_{2,i}$ are poofy for $f+1\leq i\leq n_2-1$.

Case 2 ($k=2$).  By construction, we have two neutral landmarks on each side, so there are four neutral landmarks, which means the green hyperedge $W_{2,n_2}$ is poofy. Now we show that the other green hyperedges are also poofy.
Since $k =2$, it means that $\ell_i >f$ (i.e., $\ell_i=f+1$) two times, so by construction, we have two neutrals and $n_3-2$ non-neutrals on each side.  
Since $n_2=2f+1$ is odd and $3n_2\leq 2n_3$, we have $3n_2+1\leq 2n_3$. Then 
\[n_3-2\geq \frac{3(2f+1)+1}{2}-2=3f.\]
Thus, after ignoring the two neutral landmarks on the left side, we have green coordinates $1,2,\ldots,f$ repeated at least three times, which means the green hyperedges $W_{2,i}$ are poofy for $1\leq i\leq f$. Similarly, the green hyperedges $W_{2,i}$ are poofy for $f+1\leq i\leq n_2-1$.

Case 3 ($0\leq k\leq 1$).
Since $0\leq k\leq 1$, we start by inserting one pair of neutrals, leaving $n_3-1$ non-neutral landmarks on each side. Using a calculation similar to Case 2, we have $n_3-1\geq 3f+1$.
Thus, after inserting the first pair of neutral landmarks, we have at least $3f+1$ landmarks  remaining on each side. It follows that when repeating green coordinates $1,2,\ldots,f$, we have at least four landmarks in the first green hyperedge, $W_{2,1}$, and at least three landmarks in the green hyperedges $W_{2,i}$ for $2\leq i\leq f$. Similarly, we have at least three landmarks in the green hyperedges $W_{2,i}$ for $f+1\leq i\leq n_2-1$.

We now explain why the replacement step is possible and why the green hyperedges are still poofy after replacement.
 Recall that the replacement step requires us to find a landmark of the form $(x,1,z)$ such that $x\notin\{1,2,n_1\}$ and replace it with the neutral landmark $(x,n_2,z)$. 
 We have already shown we have at least four landmarks in green hyperedge $W_{2,1}$, so a replacement will not impact the poofiness of the green hyperedge $W_{2,1}$. Additionally, if we can show all landmarks in green hyperedge $W_{2,1}$ belong to different blue hyperedges, then we will know that at least one of them does not have blue coordinate in $\{1,2,n_1\}$. 

 If $k=1$, then $\ell_1 = f+1$, but we inserted a neutral in the first blue block, so the periodic labeling of the green coordinates ensures all the landmarks in the first blue block have different green coordinates.  For the rest of the blue blocks, and in the case when $k=0$, we have $\ell_i \leq f$, so again the periodic labeling forces all landmarks in the same blue block to have different green coordinates. 
 
 Thus \cref{construction} yields pink sticks and poofy blue and green hyperedges.
\end{proof}

We next show our construction also satisfies the third condition of a basic landmark system.

\begin{proposition}\label{Lem:constructionBasic}
    Let $\mathbf{n}\in\mathcal{N}$ with $3\max(n_1,n_2)\leq 2n_3\leq n_1n_2$.
    If $W$ is constructed as in \cref{construction}, then $W$ is a basic landmark system.
\end{proposition}
\begin{proof}
   Let $\mathbf{n}\in\mathcal{N}$ with $3\max(n_1,n_2)\leq 2n_3\leq n_1n_2$.
    Suppose $W$ is constructed as in \cref{construction}.
   By \cref{Lem:onlypinksticks}, $|W_{i,a}|\geq2$ for all $1\leq i\leq 3$ and $a\in[n_i]$, so conditions (1) and (2) of \cref{def:basic} are met.
   It remains to show that hyperedges of different colors intersect in at most one vertex.
 
   Case (pink-blue). By construction, the endpoints of a pink stick belong to different blue hyperedges because their first (blue) coordinates differ by one.
   
   Case (green-pink or green-blue).
    Suppose $\alpha$ and $\beta$ are distinct landmarks with the same green coordinate. We will show $\alpha$ and $\beta$ have different pink coordinates and different blue coordinates.
    
    Case 1 ($\alpha$ and $\beta$ are neutral landmarks). If $\alpha$ and $\beta$ are neutral, we are in the case where $n_2$ is odd. Let $k$ be the number of times that $\ell_i>f$.
    
    Subcase 1.1 ($0\leq k\leq 1$). There are only three neutral landmarks: $(1,n_2,1)$, $(2, n_2, 2)$\footnote{The first coordinate of this landmark must be $2$ because the landmark is on the right side and $\ell_1\geq 2$ by Part (a) of \cref{lem:f+1}.}, and $(x, n_2, z)$, where $x\notin\{1,2,n_1\}$ by construction.  Thus the three landmarks have different blue coordinates.  The restriction on $x$ forces $z\notin\{1,2\}$. Thus the three landmarks have distinct pink coordinates.
    
    Subcase 1.2 ($k\geq 2$).
    Let $s_1,s_2,\ldots,s_k$ be the starts of the blue blocks of $\mathcal{M}(W^L)$ for which $\ell_i=f+1$.
    Then we have neutral landmarks in columns $s_1,s_2,\ldots,s_k$ of $\mathcal{M}(W^L)$
    and in columns $s_1+1,s_2+1,\ldots,s_k+1$ of $\mathcal{M}(W^R)$. Let $1\leq p,q\leq k$ with $p\neq q$. Since $f+1\geq 2$,
    we know $\{s_p,s_p+1\}$ and $\{s_q,s_q+1\}$ are disjoint.  This ensures all the neutral landmarks have different pink coordinates.
    The neutral landmarks in columns $s_p$ and $s_q$ belong to different blue blocks by construction.
    The blue coordinates of the neutral landmarks in columns $s_p$ and $s_p+1$ differ by one (keeping in mind $f+1\geq 2$).  Finally, consider neutral landmarks in columns $s_p$ and $s_q+1$. By Part (c) of \cref{lem:f+1}, multiplicity $f+1$ does not appear consecutively, so the neutral landmarks have different blue coordinates.

    Case 2 ($\alpha$ and $\beta$ are non-neutral).
    The landmarks $\alpha$ and $\beta$ must be on the same side, so they certainly have different pink coordinates. If $n_2$ is even, then $\alpha$ and $\beta$ have different blue coordinates due to the periodic labeling of the green coordinates and the fact that $\ell_i\leq f$ by Part (b) of \cref{lem:f+1}. If $n_2$ is odd, then, due to the periodic labeling before insert/replace, the only way $\alpha$ and $\beta$ can have the same blue coordinate is if they belong to a blue block with $\ell_i=f+1$. However,
    after inserting neutrals, all landmarks within the same blue block have distinct green coordinates.
    
    Thus the third condition of a basic landmark system is also met. 
\end{proof}    

We now show that the construction avoids the forbidden subgraphs.

\begin{lemma}\label{Lem:NoSquares}
    Let $\mathbf{n}\in\mathcal{N}$ with $3\max(n_1,n_2)\leq 2n_3\leq n_1n_2$.
    If $W$ is constructed as in \cref{construction}, then the landmark graph $\mathcal{G}(W)$ avoids bad $4$-cycles.   
\end{lemma}
\begin{proof}
    Recall that in a bad $4$-cycle, opposite edges of the same color must be sticks.  \cref{Lem:onlypinksticks} tells us that all blue and green hyperedges are poofy,
    so it suffices to show
    that the construction avoids bad $4$-cycles with two pink sticks. 
    
    Suppose the landmark graph contains landmarks $w_1,w_2,w_3,w_4$ such that $w_1w_2$ and $w_3w_4$ are pink sticks and $w_2$ and $w_3$ are contained in the same blue hyperedge. We will show $w_1$ and $w_4$ are in two different green hyperedges.
    We consider cases, depending on whether $w_2$ and $w_3$ are on the same side or not.
    
    If $w_2$ and $w_3$ are on the same side (left/right) then $w_1$ and $w_4$ will be on the same side (right/left).  Moreover, $w_1$ and $w_4$ are in the same blue hyperedge as each other on that side. By \cref{Lem:constructionBasic}, blue and green hyperedges intersect in at most one vertex, so we can conclude $w_1$ and $w_4$ are in different green hyperedges. 

    If $w_2$ and $w_3$ are on opposite sides (e.g., $w_2$ on the right and $w_3$ on the left), then $w_1$ and $w_4$ will be on opposite sides.
    The only way $w_1$ and $w_4$ could be in the same green hyperedge is if they are neutral because that is the only green coordinate that is on both sides.  Recall $k$ is the number of times that $\ell_i >f$.

    Case ($0\leq k\leq 1$): In this case, our neutral landmarks are $(1,n_2,1)$, $(2, n_2, 2)$ and $(x, n_2, z)$ where $x \notin \{1,2,n_1\}$.     
    For $w\in W_{2,n_2}$, let $w^*$ denote the other endpoint of the pink stick containing $w$. We need to show $(2,n_2,2)^*$ is not in the same blue hyperedge as
    $(1,n_2,1)^*$ or $(x,n_2,z)^*$. That is, $1\notin\{2,x+1\}$. Note that $x+1\neq 1$ since $x\neq n_1$.

    Case ($k\geq 2$):  In this case, we use that multiplicity $f+1$ does not occur consecutively by \cref{lem:f+1}.
    Say $w_2\in W^R$ and $w_3\in W^L$, each with blue coordinate $i$. Then $w_1\in W^L$ with blue coordinate $i-1$ and $w_4\in W^R$ with blue coordinate $i+1$, by construction.
    If $w_1$ and $w_4$ are neutral they would come from different inserts, $w_1$ from an insert at the start of blue block $i-1$ and $w_4$ from an insert at the start of blue block $i$, so $\ell_{i-1}=\ell_i=f+1$, which is a contradiction. Analogous reasoning holds if $w_2\in W^L$ and $w_3\in W^R$.

    Thus $w_1$ and $w_4$ are in different green hyperedges in all cases. It follows that there are no bad $4$-cycles in the landmark graph $\mathcal{G}(W)$.
\end{proof}
   
Finally, the construction gives a resolving set, and we can add a triple loop.

\begin{corollary}\label{resolve}
   Let $\mathbf{n}\in\mathcal{N}$ with $3\max(n_1,n_2)\leq 2n_3\leq n_1n_2$. If $W\subseteq V(K(\mathbf{n}))$ is constructed as in \cref{construction},  then $W$ resolves the direct product graph $K(\mathbf{n})$ and the triple-looped landmark system $W \cup \{u\}$ resolves the graph $K(\mathbf{n}+\mathbf{1})$.
\end{corollary}

\begin{proof}
    By \cref{Lem:constructionBasic}, we know $W$ is a basic landmark system.
    Shark teeth, rainbow $2$-$2$-triangles, and plain hex are avoided since they require sticks of different colors, but we only have pink sticks by \cref{Lem:onlypinksticks}. 
    There are no bad $4$-cycles by \cref{Lem:NoSquares}.
    The hypotheses of \cref{lem:poofyverboten} and \cref{thm:poofyloopyverboten} are satisfied.
\end{proof}

Finally, this construction gives a minimum resolving set which gives us the metric dimension of the infinite family of graphs $K(\mathbf{n+1})$ for $\mathbf{n}$ in the middle cone.

\begin{theorem}\label{mdimthm}
 Let $\mathbf{n}\in\mathcal{N}$.
    If
    $3\max(n_1,n_2) \leq 2n_3\leq n_1n_2$, then the direct product graph $K(\mathbf{n}+\mathbf{1})=K_{n_1+1}\times K_{n_2+1}\times K_{n_3+1}$ has metric dimension $2(n_3+1)-1$.
\end{theorem}

\begin{proof}
    By \cite[Theorem 2]{FGUhl}, $\dim(K(\mathbf{n}+\mathbf{1}))\geq 2(n_3+1)-1$.
    By \cref{resolve}, $\dim(K(\mathbf{n}+\mathbf{1}))\leq 2n_3+1$.
\end{proof}
%%%%%%%%%%%%%%%%%%%%%%%%%%%%%%%%%%%%%%%%%%%%

Additionally, our resolving sets are total-dominating sets. 
A \standout{dominating set} is a set of vertices such that every vertex of the graph is either in the set or adjacent to a vertex in the set. A \standout{total-dominating set} is a set of vertices such that every vertex of the graph is adjacent to a vertex in the set.
A \standout{locating-(total)-dominating set} is a resolving set that is also a (total)-dominating set. The \standout{location-(total)-domination number} of a graph is the minimum size of a locating-(total)-dominating set. 

The location-domination number of a direct product of two complete graphs was determined in \cite{Junnila2025newoptimalresultscodes}. We next show the resolving sets from our construction are total-dominating sets, so
we have the location-total-domination number of the family of direct products of three complete graphs $K(\mathbf{n+1})$ when $\mathbf{n}$ is in the middle cone.

\begin{lemma}\label{lem:totaldomination}
    Let $\mathbf{n}\in\mathcal{N}$ with $3\max(n_1,n_2)\leq 2n_3\leq n_1n_2$. If $W\subseteq V(K(\mathbf{n}))$ is constructed as in \cref{construction},  then $W$ is a total-dominating set for the direct product graph $K(\mathbf{n})$, and the triple-looped landmark system $W \cup \{u\}$ is a total-dominating set for the graph $K(\mathbf{n}+\mathbf{1})$.
\end{lemma}
\begin{proof}
Let $W$ be a basic landmark system for $K(\mathbf{n})$ constructed using \cref{construction}.
Let $\alpha=(a_1,a_2,a_3)$ be a vertex in $K(\mathbf{n})$. The footprint of $\alpha$ consists of three hyperedges of different colors $W_{1,a_1}\cup W_{2,a_2}\cup W_{3,a_3}$. We want to show there is some landmark in $W$ that is not in the footprint (because that would mean it has no coordinates in common with $\alpha$ and hence is adjacent to $\alpha$ in $K(\mathbf{n})$). We will do this by showing $|W|>|W_{1,a_1}\cup W_{2,a_2}\cup W_{3,a_3}|$.

Note that $|W_{1,a_1}|\leq n_2$ and $|W_{2,a_2}|\leq n_1$ because otherwise we would get a contradiction to \cref{def:basic} part (3). Additionally, $|W_{3,a_3}|=2$ by \cref{Lem:onlypinksticks}. 
So $|W_{1,a_1}\cup W_{2,a_2}\cup W_{3,a_3}|\leq n_1+n_2+2$.
Also, $2n_3\geq 3\max(n_1,n_2)>n_1+n_2+2$. Hence, by sizes, there exists a vertex in $W$ but not in the footprint of $\alpha$. Thus $W$ is a total-dominating set for the graph $K(\mathbf{n})$.

Next, we show that $W\cup\{u\}$ is a total-dominating set for the graph $K(\mathbf{n+1})$, where $u=(n_1+1,n_2+1,n_3+1)$. As above, each $\alpha$ in $V(K(\mathbf{n}))$ is adjacent to a vertex in $W\subseteq W\cup\{u\}$.
For $\alpha\in V(K(\mathbf{n+1}))-V(K(\mathbf{n}))$, we still have $|W_{1,a_1}\cup W_{2,a_2}\cup W_{3,a_3}|\leq n_1+n_2+2$ because the only new hyperedges are loops. Note that $|W\cup\{u\}|=2n_3+1$. 
So we have
$|W \cup \{u\}|>|W_{1,a_1}\cup W_{2,a_2}\cup W_{3,a_3}|$. Thus there exists a landmark that is not in the footprint of $\alpha$. Hence $W\cup\{u\}$ is a total-dominating set for the graph $K(\mathbf{n+1})$.
\end{proof}

\begin{theorem}
     Let $\mathbf{n}\in\mathcal{N}$. If
    $3\max(n_1,n_2) \leq 2n_3\leq n_1n_2$, then for the direct product graph $K(\mathbf{n}+\mathbf{1})=K_{n_1+1}\times K_{n_2+1}\times K_{n_3+1}$, the metric dimension, location-domination, and total-location-domination numbers are all equal to $2(n_3+1)-1$.
\end{theorem}
\begin{proof}
    Let $G=K(\mathbf{n+1})$, and let $\dim(G)$ denote metric dimension, $\gamma^{L}(G)$ denote the location-domination number, and $\gamma^L_{t}(G)$ denote the location-total-domination number.
    Note that every total-dominating set is dominating and every locating-dominating set is resolving, so
    \[\dim(G)\leq \gamma^L(G)\leq \gamma^L_t(G).\]
    Since $W\cup\{u\}$ is a minimum resolving set (by \cref{mdimthm}) and a total-dominating set (by \cref{lem:totaldomination}), we get equality of all three numbers for the graph $K(\mathbf{n+1})$.
\end{proof}

%%%%%%%%%%%%%%%%%%%%%%%
\section{Upper and Lower Cones}\label{sec:conclusion}
Having addressed all of the direct products of three complete graphs stemming from the middle cone, we are left with asking about the upper and lower cones.  

\subsection{Upper Cone} The techniques using basic landmark systems will not work for the upper cone.

\begin{proposition}\label{notpossible}
  Let $\mathbf{n}\in\mathcal{N}$. If $2n_3> n_1n_2$, then the direct product graph $K(\mathbf{n})$ cannot be resolved with a \textbf{basic} landmark system of $2n_3$ vertices. 
\end{proposition}
\begin{proof}
  Suppose $W$ is a basic landmark system for $K(\mathbf{n})$ that has $2n_3$ elements. Each landmark belongs to some intersection $W_{1,a_1}\cap W_{2,a_2}$ with $1\leq a_i\leq n_i$. There are $n_1n_2$ possible intersections, but since $W_{1,a_1}$ and $W_{2,a_2}$ are hyperedges of different colors, each intersection contains at most one landmark. Hence there are at most $n_1n_2$ landmarks in $W$, and if $2n_3>n_1n_2$, there are no basic landmark systems with $2n_3$ vertices.
\end{proof}

Note that \cref{notpossible} does not imply that the metric dimension is greater than $2n_3$ in the upper cone, simply that a resolving set of size $2n_3$ cannot be achieved using a basic landmark system. 

It is still an open question as to which graphs in the upper cone can be resolved in $2n_3$ or $2n_3-1$ by loosening the restrictions (e.g., allowing hyperedges of different colors to intersect in more than one vertex or having some of the $W_{i,a}$ empty). We do know such graphs exist, as the next example shows.

\begin{example}[Upper cone]
In \cref{fig:landmarkgraph338} we have a landmark graph for a resolving set of $K(3,3,8)$ in which hyperedges of different colors intersect in more than one vertex.  There are $16$ landmarks, so $15\leq \dim(K(3,3,8))\leq 16$.
\end{example}

\begin{figure}
    \centering
\begin{tikzpicture}[scale=1]
    \draw[blueedge,double distance=60pt, line cap=round] (0,-1.5)--(0,4.5);
    \draw[blueedge,double distance=60pt, line cap=round] (6,-1.5)--(6,4.5);
    \draw[blueedge,double distance=60pt, line cap=round] (3.0,-1.5)--(3.0,4.5);
   
    \draw[greenedge] (-1.5,4.75)--(7.5,4.75);    \draw[greenedge] (-1.5,3.25)--(7.5,3.25);
    \draw[greenedge]   (-1.5,3.25) to[out=180,in=180] (-1.5,4.75);
    \draw[greenedge]   (7.5,3.25) to[out=0,in=0] (7.5,4.75);
    
    \draw[greenedge] (0,2.25)--(6,2.25);
    \draw[greenedge] (0,.75)--(6,.75);
    \draw[greenedge]   (0,.75) to[out=180,in=180] (0,2.25);
    \draw[greenedge]   (6,.75) to[out=0,in=0] (6,2.25);

     \draw[greenedge] (-1.5,-.25)--(7.5,-.25);
    \draw[greenedge] (-1.5,-1.75)--(7.5,-1.75);
    \draw[greenedge]   (-1.5,-1.75) to[out=180,in=180] (-1.5,-.25);
    \draw[greenedge]   (7.5,-1.75) to[out=0,in=0] (7.5,-.25);

    \draw[pinkedge] (3,4)--(3,2);
    \draw[pinkedge] (3,1)--(3,-1);
    
    \draw[pinkedge] (-0.75,4.)--(-0.75,-1);
    \draw[pinkedge] (0.4,4)--(0.4,2);
    \draw[pinkedge] (0.4,1)--(0.4,-1);
    
    \draw[pinkedge] (6.75,4)--(6.75,-1);
    \draw[pinkedge] (5.6,4)--(5.6,2);
      \draw[pinkedge] (5.6,1)--(5.6,-1);

    \draw[fill=black] (-0.75,4) circle (2.5pt) node[above] {};
    \draw[fill=black] (0.4,4) circle (2.5pt) node[below] {};
    \draw[fill=black] (0.4,2) circle (2.5pt) node[below] {};
     \draw[fill=black] (0.4,1) circle (2.5pt) node[below] {};
    
    \draw[fill=black] (0.4,-1) circle (2.5pt) node[below] {};
    \draw[fill=black] (-.75,-1) circle (2.5pt) node[below] {};
      
    \draw[fill=black] (3.,4) circle (2.5pt) node[above] {};
    \draw[fill=black] (3,2) circle (2.5pt) node[below] {};
    \draw[fill=black] (3,1) circle (2.5pt) node[below] {};
    \draw[fill=black] (3,-1) circle (2.5pt) node[below] {};

    \draw[fill=black] (6.75,4) circle (2.5pt) node[above] {};
     \draw[fill=black] (5.6,4) circle (2.5pt) node[above] {};
    \draw[fill=black] (5.6,2) circle (2.5pt) node[below] {};
     \draw[fill=black] (5.6,1) circle (2.5pt) node[below] {};
    \draw[fill=black] (5.6,-1) circle (2.5pt) node[above] {};
    \draw[fill=black] (6.75,-1) circle (2.5pt) node[below] {};
    
\end{tikzpicture}
    \caption{Landmark graph $\mathcal{G}(W)$ of a resolving set of size $16$ for $K(3,3,8)$.}
    \label{fig:landmarkgraph338}
\end{figure}
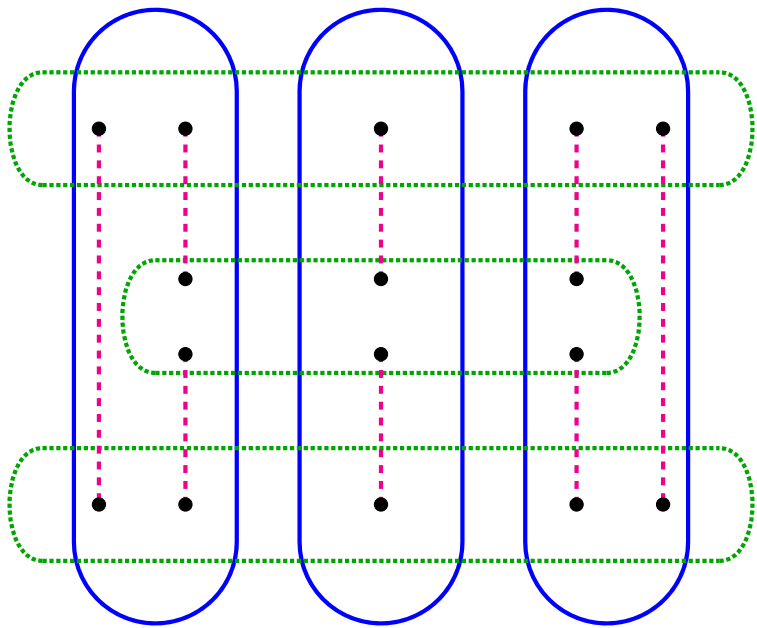

\subsection{Lower Cone} 
Finally, what about the lower cone? If $2n_3<3\max(n_1,n_2)$, then any basic landmark system of size $2n_3$ that resolves $K(\mathbf{n})$ must have sticks of multiple colors due to the pigeonhole principle. Consequently, design of a resolving basic landmark system requires more care to avoid bad $4$-cycles, plain hex, and shark teeth. 

It is an open question as to which graphs in the lower cone can be resolved in $2n_3$ or $2n_3-1$. We know some examples exist. 
For instance, despite not coming from the middle cone, the graph $K(6,7,7)$ can be resolved with $13$ landmarks by using \cref{construction} to resolve $K(5,6,6)$ and adding a triple loop. This approach is not always successful though, and the challenge is finding unified and easy to describe constructions.

\subsection*{Code} The examples mentioned in this section can be checked by computer. Our code for verifying resolving sets of $K(\mathbf{n})$ with SageMath \cite{SageMath} is available 
at \href{https://github.com/fostergreenwood/metric-dimension}{https://github.com/fostergreenwood/metric-dimension}.

%%%%%%%%%%%%%%%%%%%%%%%%%%%%%%%%%%%%%%%%%%%
%%%%%%%%%%%%%%%%%%%%%%%%%%%%%%%%%%%%%%%%%%%%%%%%%%%%%%%%%%%%%%%%%
\bibliographystyle{alphaurl}
\bibliography{metricdimension2}

\end{document}